\documentclass{amsart}

\newcommand{\Ci}{\mathscr{C}}
\newcommand{\Sone}{\mathbb{S}^1}

\newcommand{\R}{\mathbb{R}}

\newcommand{\F}{\mathcal{F}}

\newcommand{\vertiii}[1]{{\left\vert\kern-0.25ex\left\vert\kern-0.25ex\left\vert #1 
    \right\vert\kern-0.25ex\right\vert\kern-0.25ex\right\vert}}

\newcounter{tmp}

\usepackage{amsmath, amssymb}
\usepackage{amsthm}

\usepackage[initials]{amsrefs}
\usepackage{amsmath,amsthm,amscd}
\usepackage{enumerate}

\usepackage{amssymb}
\usepackage{mathrsfs}
\usepackage{graphicx}

\usepackage{color}
\usepackage{mathtools}
\mathtoolsset{showonlyrefs=true}

\usepackage[T1]{fontenc}

\usepackage{multirow}

\usepackage{amsmath,amsthm,amssymb,amscd}
\usepackage{enumerate}

\usepackage{amssymb}
\usepackage{mathrsfs}
\usepackage{graphicx}

\theoremstyle{plain}
\newtheorem{thm}{Theorem}
\newtheorem{cor}[thm]{Corollary}
\newtheorem{lem}[thm]{Lemma}
\newtheorem{prop}[thm]{Proposition}

\newtheorem{problem}[thm]{Problem}

\theoremstyle{definition}
\newtheorem{dfn}[thm]{Definition}
\newtheorem{example}[thm]{Example}

\theoremstyle{remark}
\newtheorem{remark}[thm]{Remark}

\AtBeginDocument{%
   \def\MR#1{}
}

\numberwithin{equation}{section}

\begin{document}

\title{Existence and non-existence of length averages for foliations}

\author{Yushi Nakano}
\address{Department of Mathematics, Tokai University,  4-1-1 Kitakaname, Hiratuka,
Kanagawa, 259-1292, JAPAN}
\email{yushi.nakano@tsc.u-tokai.ac.jp}

\author{Tomoo Yokoyama}
\address{Department of Mathematics, Kyoto University of Education/JST Presto, 
1 Fujinomori, Fukakusa, Fushimi-ku Kyoto, 612-8522, Japan}
\email{tomoo@kyokyo-u.ac.jp}

\thanks{The first author is partially supported by JSPS KAKENHI Grant Numbers  17K05283.
The second author is partially supported by JST PRESTO Grant Number JPMJPR16ED at Department of Mathematics, Kyoto University of Education.}

\begin{abstract}
Since the pioneering work of Ghys, Langevin and Walczak among others, it has been known that several methods of dynamical systems theory can be adopted to study of foliations. Our aim in this paper is to investigate complexity of foliations, by generalising existence problem of  time averages in dynamical systems theory to foliations: It has recently been realised that a positive Lebesgue measure set of points without time averages only appears for complicated dynamical systems, such as dynamical systems with heteroclinic connections or homoclinic tangencies. In this paper, we introduce the concept of length averages to singular foliations, and attempt to collect interesting examples with/without length averages. In particular, we demonstrate that  length averages exist everywhere for any codimension one $\Ci^1 $ orientable singular foliation  without degenerate  singularities on a  compact surface under a mild  condition on quasi-minimal sets of the foliation, which is in strong contrast to time averages of surface flows. 
\end{abstract}

\maketitle

\section{Introduction}\label{section:introduction}
There are plenty of papers in which several methods of dynamical systems theory were successfully adopted to study of foliations. One of the pioneering work was done by Ghys, Langevin and Walczak \cite{Ghys1988}, in which entropies of foliations were first introduced (a close link between the entropy of foliations and the Godbillon-Vey class of foliations can be found in unpublished earlier works by
 Duminy, see a paper \cite{CC2002} and the preface of \cite{Walczak2012}). A standard reference for connection between the topology and  dynamics of foliations is a monograph of Walczak \cite{Walczak2012}.

Our aim in this paper is to investigate ``complexity'' of foliations, by generalising existence problem 
 of time averages in dynamical systems theory to foliations. 
We first recall time averages for flows. 
A $\Ci ^r$ \emph{flow $F$}  on  a smooth manifold  $M$ with  $r\geq 1$  is given as  a $\Ci ^r$ mapping from $\R \times M$ to $M$ such that $f^t\equiv F(t,\cdot)$ is a $\Ci ^r$ diffeomorphism  for each $t\in \R$, and that  $f^0=\mathrm{id} _M$ and $f^{s+t}=f^s\circ f^t$ for each $s,t \in \R$. 
For each point $x\in M$ and continuous function $\varphi : M \to \mathbb R$, we call
\begin{equation}\label{def:hb}
\lim _{T\to \infty}\frac{1}{T} \int ^T_0 \varphi (f^t(x)) dt
\end{equation}
the \emph{time average} of $\varphi$ at $x$. 
By Birkhoff's ergodic theorem (cf.~\cite{KH95}), if $\mu$ is an invariant measure of $F$,  then 
 the time average of  any continuous function exists $\mu$-almost everywhere.
Hence, when the time average of some continuous function does not exist  at a point,  the point is called a \emph{non-typical point} or an  \emph{irregular point} 
 (see e.g.~\cites{BS2000, Thompson2010}; 
such a point is also called a point with \emph{historic behaviour} in another context \cites{Ruelle2001, Takens2008, KS2017, LR2017}).

There are a wide variety of examples $F$ for which time averages exist for any continuous function and Lebesgue almost every point. 
(We notice that it is rather special  that time averages exist at \emph{every} point; see e.g.~
\cites{BS2000, Thompson2010}.) 
The simplest example is a conservative dynamical system (i.e.~a dynamical system for which a Lebesgue measure is invariant), due to Birkhoff's ergodic theorem. 
Another famous example may be a $\Ci ^r$ Axiom A flow on a compact smooth Riemannian manifold with no cycles for $r> 1$
(\cites{BR1975, Ruelle1986}). 
It is also known that time averages  exist Lebesgue almost everywhere  for large classes of non-uniformly hyperbolic dynamical systems \cites{Alves2000, ABV2000}. 
Moreover, from classical theorems by Denjoy and Siegel (\cite{Tam}*{\S 1}; see also \cite{KH95}*{\S 11} and Subsection \ref{subsection:sf}), we can immediately conclude  that  if $F$ is a  $\Ci ^r$ flow on a torus 
 with no singular points for $r\geq1$, then time averages exist at 
(not only Lebesgue almost every point but also)
 every point in the torus.

On the other hand, it also has been known  that for several (non-hyperbolic) dynamical systems, the time average of some continuous function does not exist on a Lebesgue positive measure set.  
One of the most famous example is Bowen's folklore example, which is a smooth surface flow 
 with two heteroclinically connected singular points 
(refer to \cite{Takens1994}; see also Remark \ref{rmk:Bowen} and Figure \ref{fig0}). 
One can also find another interesting example without time averages 
constructed by Hofbauer and Keller \cite{HK1990} for some quadratic maps.
Furthermore,  recently Kiriki and Soma \cite{KS2017} showed that there is a locally dense class of $\Ci ^r$  surface diffeomorphisms for which some time average does not exist on  a positive Lebesgue measure  set, by employing dynamical systems with homoclinic tangencies 
%
(for a $\Ci ^r$ diffeomorphism $P$, the time average of a continuous  function $\varphi$ at a point $x$ is given by $\lim _{n\to \infty} 1/n \sum _{j=0}^{n-1} \varphi (P^j(x))$; their result was further extended to  three-dimensional flows in \cite{LR2017}). 

From that background we here consider length averages for $\Ci ^r$ singular foliations with $r\geq 1$.
A standard reference of $\Ci ^r$ (regular) foliations is \cites{HH1986A,HH1987B,CC2003I,CC2003II}. 
We refer to Stefan \cite{Stefan1974} and Sussmann \cite{Sussmann1973} for the definition of singular foliations on  a smooth manifold $M$ 
without boundary. 
We here  recall that it is shown in \cite{Kubarski1990} that when $r\geq 1$, a $\Ci ^r$ singular foliation $\F$ on a smooth manifold $M$ 
without boundary
 is equivalent to a partition into immersed connected submanifolds  such that for any point $x \in M$, there is a fibred chart at $x$ with respect to $\mathcal F$ (see \cite{Kubarski1990} for the definition of fibred charts; note that fibred charts are not well defined if $M$ has boundary), so that each element $L$ of $\mathcal{F}$  (called a \emph{leaf}) admits a dimension, denoted by $\dim  (L)$. 
Furthermore, we say that  a partition $\mathcal F$ on a smooth manifold $M$ with boundary $\partial M$ is a $\Ci ^r$ singular foliation  if 
the induced partition $\hat{\mathcal F}$ of $\mathcal F$ on 
the double of $M$ (i.e.~$M\times \{0,1\}/\sim $, where $(x,0)\sim (x,1)$ for $x\in \partial M$; see Figure \ref{pic01}) is a $\Ci ^r$ singular foliation 
(the pair $(M,\mathcal F)$ is  called a $\Ci^r$-\emph{foliated manifold}), and define the dimension 
 of each leaf  $L\in \mathcal F$ through a boundary point as the dimension of the lifted  leaf of  $L$ in $\hat{\mathcal F}$. 
%
The integer $\dim  (\mathcal F) := \max \{ \dim  (L) \mid L \in \mathcal F\}$ is called the \emph{dimension} of $\mathcal{F}$, and $\dim (M) -  \dim  (\mathcal F)$  the \emph{codimension} of $\mathcal{F}$. 
A leaf $L$ is said to be \emph{regular} if $\dim  (L) =\dim  (\mathcal F)$ and \emph{singular} if $\dim  (L) <\dim  (\mathcal F)$. 
The union of singular leaves is called the \emph{singular  set} and  denoted by $\mathrm{Sing} (\mathcal F)$. With a usual abuse of notation, we simply say that  $\mathcal F$ is a \emph{regular foliation}  if $\mathrm{Sing} (\mathcal F ) =\emptyset $.

When $\dim  (\mathcal F)=1$ (in particular, when $\mathcal F$ is a codimension one singular  foliation $\F$ on a surface $M$), we have that $\mathrm{Sing} (\mathcal F) = \bigcup _{L\in \mathcal S} L$ with $\mathcal S = \{ L \in \F \mid \# L  = 1 \}$, so that $\F - \mathcal S$ is a $\Ci ^r$ regular foliation of $M - \mathrm{Sing} (\mathcal F)$. 
A point in $\mathrm{Sing} (\mathcal F) $ is called a \emph{singularity} or a \emph{singular point}. 
A singularity $x$ of $M$ is said to be (metrically) \emph{non-degenerate} if there are a neighbourhood $U$ of $x$ and a $\Ci ^1$ vector field $A$ on $M$ such that the set of orbits of the flow generated by $A$  corresponds to $\F$ on $U$ and that $x$ is a non-degenerate singularity  of $A$ (i.e.~$A(x)=0$ and  each eigenvalue of $DA$ at $x$ is non-zero). 
Otherwise, a singularity is said to be (metrically) \emph{degenerate}. 
Notice
that, by definition, each boundary component of a foliation with no degenerate
singular points on a surface is either a circle which is transverse to the foliation
or a union of leaves. In particular, each center does not belong to the boundary.
We also note that if there is no degenerate singularities, then there are at most
finitely many non-degenerate singularities. In other words, any accumulation point
of infinitely many non-degenerate singularities need to be degenerate. 

We denote by $\mathcal{F}(x)$ the leaf of a foliation $\mathcal{F}$ through a point $x\in M$. 
Let $d$ be the distance on leaves of $\mathcal F$ induced by a Riemannian metric of $M$.\footnote{
Existence
of length averages (given in Definition \ref{def:HBF}) 
is independent of the choice of Riemannian structures when the manifold is compact, because any two Riemannian metrics on a compact manifold are Lipschitz equivalent (for two metrics $g_1$ and $g_2$   on $M$, a mapping $ v \mapsto  \frac{g_1(v, v )}{g_2 (v ,v )}$ on the unit tangent bundle of $M$ is continuous and strictly positive, and by compactness, it is bounded above and below by positive constants). 
On the other hand, it can depend on the choice of Riemannian structures when the manifold is
non-compact, as indicated in Example \ref{ex2} and \ref{ex3}.} 
\begin{dfn}\label{def:HBF}
For a point $x\in M$ and  a continuous function $\varphi$ on $M$, we define   the \emph{length average} of $\varphi$ at 
$x$ by
\begin{equation}\label{eq:hb2}
\lim _{r\to \infty} \frac{1}{\left\vert B_r^{\mathcal F}(x)\right\vert } \int _{B_r^{\mathcal F}(x)} \varphi (y)dy,
\end{equation}
where $ B_r^{\mathcal F}(x)=\{ y\in \mathcal{F}(x) \mid d(x,y) <r\}$ and $\left\vert B_r^{\mathcal F}(x)\right\vert = \int _{B_r^{\mathcal F}(x)} dy$ is a $p$-dimensional volume of $ B_r^{\mathcal F}(x)$ with $p=\mathrm{dim} (\mathcal{F}(x))$.
If  the length average of $\varphi$ at $x$ exists for every continuous function $\varphi$, then we simply say that length averages exist at $x$.
\end{dfn}

We emphasise that the length average of $\varphi$ for the foliation generated by a flow does not coincide with the time average of $\varphi$ for the flow in general (see Remark \ref{rmk:1} for details). 
On the other hand, it also should be noticed that these two averages coincide with each other in the regular case 
(note that we can reparametrise the flow into a flow with unit velocity, because the compactness of the surface implies that the velocities at any points of the flow are bounded and away from zero).  
Finally, we would like to say that we are strongly   interested in  any physical meaning and application of  length averages.

\subsection{Foliations without length averages}\label{subsection:1.1}
For foliations on open manifolds or foliations with 
 degenerate singular points,
  there are abundant examples for which some length average does not exist.  
In Section \ref{example}, we will provide with the following examples without length averages: 
\begin{itemize}
\item Trivial foliation on a Euclidean space (Example \ref{ex1}), and foliation by the (regularised) Koch curves on an open disk  (Example \ref{ex2}). 
\item  Foliation by the  (regularised) Koch curves on a compact surface with one 
 degenerate singularity  (Example \ref{ex3}). 
\end{itemize}
These examples will make clear that 
the  \emph{compactness} of $M$ and the \emph{non-degeneracy} 
 of singularities of $\mathcal F$ might be necessary for investigating  interesting examples without length averages. Moreover, in  higher codimensional cases, we can easily construct foliations without length averages.  
We  will have the following example:
\begin{itemize}
\item Foliation generated by the suspension flow of any diffeomorphism $P$ without time averages (that is, $P$ admits a positive Lebesgue measure set consisting of points without time averages, such as the time-one map of  the Bowen's surface flow, see Remark \ref{rmk:Bowen} and Figure \ref{fig0}). We refer to Example \ref{ex4}. 
Since time averages exist everywhere  for any one-dimensional diffeomorphism (see \cite{KH95}*{\S 11} e.g.), the dimension of the phase space of $P$ has to be $\geq 2$, so that this is a
\emph{codimension two} foliation of a compact manifold with no  singularities.
\end{itemize}

\subsection{Foliations with length averages}
In contrast to the examples in the previous subsection, 
length averages  exist at \emph{every} point for 
any codimension one $\Ci ^1$  
 singular foliation 
   on a compact surface without  
    degenerate singularities
   under a mild condition on quasi-minimal sets  (which is  a  significantly stronger conclusion than 
one can expect from 
a
straightforward dynamical analogy, as indicated in Remark \ref{rmk:1}). 
 We will also show that the conclusion holds for Lebesgue almost every point under another condition on quasi-minimal sets. 
\begingroup
\setcounter{tmp}{\value{thm}}
\setcounter{thm}{0}
\renewcommand\thethm{\Alph{thm}}
\begin{thm}\label{main-thm} 
For any codimension one  $\Ci^1$ singular foliation without  
 degenerate singularities  on a compact surface, the following holds:
\begin{itemize}
\item[$\mathrm{(1)}$]
Length averages exist everywhere
 if the union of quasi-minimal sets of the foliation is empty or uniquely ergodic (see  Definition \ref{dfn:ue2}). 
\item[$\mathrm{(2)}$]
Length averages exist Lebesgue almost everywhere if  the union of  quasi-minimal sets of the foliation is locally dense (see  Definition \ref{dfn:ue2}). 
\end{itemize}
\end{thm}
\endgroup
  We will see in Remark \ref{rmk:r121} that every  quasi-minimal set is uniquely ergodic  for ``almost every'' foliation on a compact surface, and for  every foliation on a compact surface with orientable genus less than two or non-orientable genus less than four. 

\begin{remark}\label{rmk:1}
One  can see  an essential difference between  time averages and length averages for the Bowen's classical surface flow, which heuristically explains the strong contrast between  Theorem \ref{main-thm}  and time averages  of surface flows. 
Let $\mathcal F$  be a  singular foliation given as the set  of all orbits of a continuous flow $F$, i.e.~each leaf $\mathcal{F}(x)$  of $\mathcal F$ through $x$  is given as  the orbit $O(x):=\{ f^t(x) \mid t\in \R \}$ of $x$. We call $\mathcal F$ the \emph{singular foliation generated by $F$} (such singular foliations will be intensively studied  in Section \ref{proof}).
 Let $F$ be the  Bowen's  flow   on a compact surface $S$,  and $\mathcal F$ the singular foliation generated by $F$. 
Then, it follows from    \cite{Takens1994} that there exists a positive Lebesgue measure set $D\subset S$ such that for any $x\in D$, the time average  $\lim _{T\to \infty}\frac{1}{2T} \int ^T_{-T} \varphi (f^t(x)) dt$ of some continuous function $\varphi :S \to \R$ along the orbit $O(x)$ of $x$   does not exist (by virtue of the non-existence of the time average \eqref{def:hb} of some continuous function $\varphi $ along the forward orbit $O_+(x) :=\{ f^t(x) \mid t\in \R _+\}$ and by the convergence of  $f^{-t}(x)$ to a repelling point  as $t\to \infty$).
On the other hand, since $\mathcal F$ is a codimension one singular foliation with no degenerate singularities  on the compact surface $S$ and it does not have any quasi-minimal set, 
 we can apply Theorem \ref{main-thm}: for any $x \in D$, the length average  \eqref{eq:hb2} of any continuous function  $\varphi : M \to \R$ along the leaf $\mathcal{F}(x)=O(x)$ through $x$  \emph{does} exist. 
See Remark \ref{rmk:Bowen} for detail.
Compare also with Example \ref{ex4}.
 \end{remark}

As an  application of the item (1) of Theorem \ref{main-thm} together with Remark \ref{rmk:r121}, 
we can immediately obtain the following corollary for regular foliations: 

\begingroup
\setcounter{thm}{1}
\renewcommand\thethm{\Alph{thm}}
\begin{cor}\label{main-thm01b}
Length averages exist everywhere for any codimension one  $\Ci^1 $ regular foliation   on a compact surface. 
\end{cor}
\endgroup

One may realise that Corollary \ref{main-thm01b} is rather a 
generalisation of classical theorems by Denjoy and Siegel for \emph{toral} flows without singularities 
(this is contrastive to the non-regular case as in   
 Remark \ref{rmk:1}). 
In fact, existence of time averages at every point for surface flows (in particular, toral flows) without singularities is 
a direct consequence of Corollary \ref{main-thm01b} 
by the remark after Definition \ref{def:HBF} for regular foliations. 
This similarity between flows and foliations  seems to be reminiscent of the previous works 
 for 
connection between the geometric and dynamics of foliations (such as the geometric entropy of a foliation and the topological entropy of the associated 
holonomy group  
  \cite{Ghys1988}*{Theorem 3.2}). 
Furthermore,  existence of time averages for the group generated by the ``Poincar\'e map'' of foliations plays an important role  in the proof of Theorem \ref{main-thm} (see Subsection \ref{subsection:sf}),  
and thus, 
 it is likely that analysis of holonomy group is also crucial  for analysis of length averages.
%
(Technically, moreover, the difference between time averages and length averages is closely related with the  first return time to a cross-section, see Remark \ref{rmk:Bowen}.)

According to the   examples in the previous subsection  without length averages and Theorem \ref{main-thm}, 
we may naturally  ask the following problem:
\begin{problem}\label{problem}
Does  length averages exist 
 everywhere for any  codimension one 
 foliation on a compact smooth Riemannian manifold?
\end{problem}

Theorem \ref{main-thm} is a satisfactory but not complete answer to Problem \ref{problem} for  foliations on compact surfaces. 
One may  ask whether length averages exist  everywhere without the assumption on  quasi-minimal sets in the item (1) of Theorem \ref{main-thm}.
This can be reduced to  existence problem of time averages for \emph{any} interval exchange transformations, 
  see Remark \ref{rmk:r121}
  and 
  the proof of Proposition \ref{IET2r2}. 
In higher dimensions, 
statistical behaviours of foliations  are 
more complicated   due to  the number of ends:  in surfaces the numbers of ends of foliations are either zero or two (since 1-D paracompact manifolds are either the circle or the real line), while in higher dimensional manifolds they can be infinite.  
Therefore, investigating codimension one foliations with infinite ends (e.g.~Hirsch foliations \cite{Hirsch1975} or Sacksteder foliations \cite{Sacksteder1964}) would also be a first step to the answer of Problem \ref{problem}. 
We again note that for analysis of length averages, it seems important  to understand the dynamics of associated holonomy groups, so recent development in ergodic theory of codimension-one foliations and   group actions by circle diffeomorphisms (e.g.~\cite{DKN2018})  might be helpful, 
  see also Example \ref{ex4} and Remark \ref{rmk:ex4}. 
%

\begin{remark}\label{rmk:regularity}
 %
%
It is natural to ask whether or not further dynamical analogies hold for length averages of foliations, such as a version of Birkhoff's ergodic theorem for length averages. 
 Furthermore, 
it seems  interesting to investigate a generalisation of  Theorem \ref{main-thm} (and Corollary \ref{main-thm01b}) to $\Ci ^0$ singular foliations in an appropriate sense. 
We assumed singular foliations in Theorem \ref{main-thm} to be $\Ci ^1$,  partly because $\Ci ^1$ makes it possible to define  the length $r$ in each leaf, which is needed to make sense Definition \ref{def:HBF}. 
It is not difficult to  define length averages for \emph{piecewise} $\Ci ^1$ singular foliations, including foliations by the classical Koch curves (see Example \ref{ex3}).
Possibly, one can even define length averages for $\Ci ^0$ singular foliations if one considers, instead of a length $r$ and  the ball $B_r^{\mathcal F}(x)$ in $\mathcal{F}(x)$, a number $N$ and the set of plaques $\ell $ of $\mathcal{F}(x)$ satisfying that  there are $N$    plaques  $ \tilde \ell _1, \ldots , \tilde \ell _N$ such that  $x\in \tilde \ell _1$, $\ell = \tilde \ell _N $ and $\tilde \ell _j\cap \tilde \ell _{j+1} \neq \emptyset$ (and takes the infimum  over all foliated atlases, if necessary). 
We think that  this definition would help connecting length averages for foliations to time averages for  their  ``Poincar\'{e}  map'', as  in   Example \ref{ex4}.
We also note that in  the argument of the proof of Theorem \ref{main-thm} and Corollary \ref{main-thm01b},  the $\Ci ^1$ regularity is (essentially) necessary only in Lemma \ref{lem:boundedlength}. 
%
%
\end{remark}

\section{Proof of Theorem \ref{main-thm} and Corollary \ref{main-thm01b}}\label{proof}

%


\subsection{Singular foliations generated by flows}\label{subsection:sf}
We shall deduce Theorem \ref{main-thm}  from the following theorem for length averages of singular foliations generated by 
 surface flows, which is our goal in this subsection.
\begin{thm}\label{main-thm02}
 For any $\Ci ^1$ singular foliation generated by a $\Ci ^1$ flow without 
  degenerate singular points on an orientable compact surface, the following holds:
\begin{itemize}
\item[$\mathrm{(1)}$] Length averages exist   everywhere  if the union of quasi-minimal sets of the flow is empty or uniquely ergodic (see Definition \ref{dfn:ue}). 
\item[$\mathrm{(2)}$] Length averages exist  Lebesgue almost everywhere if  the union of  quasi-minimal sets of the flow is locally dense (see Definition \ref{dfn:ue}). 
\end{itemize}
\end{thm}

Let $F$ be  a $\Ci ^1$ flow without 
 degenerate singularities on an orientable  compact surface $S$. 
An orbit of a point $x$ is called \emph{recurrent} if $x \in \omega(x) \cup \alpha(x)$, where $\omega(x) := \bigcap_{s\in \mathbb{R}}\overline{\{f^t(x) \mid t > s\}}$ and $\alpha(x) := \bigcap_{s\in \mathbb{R}}\overline{\{f^t(x) \mid t < s\}}$ (called  \emph{$\omega$-limit set} and  \emph{$\alpha$-limit set} of $x$, respectively).
Recall that a  \emph{quasi-minimal set} is an orbit closure of a non-closed recurrent orbit. 
We say that a subset of $S$ is a \emph{circuit}  if it is  an image of a circle composed of a finite number of singularities together with homoclinic and heteroclinic orbits connecting these singularities.
A circuit $\gamma$  is said to be  \emph{attracting} if there is a continuous mapping $q$ from $[0,1] \times \mathbb{S}^1$ to a  neighbourhood of $\gamma$ such that $q(0,\cdot ) $ is a continuous mapping from the unit circle $\Sone$ to the circuit, and that  $q((0,1] \times \mathbb{S}^1)$ is an  embedded annulus satisfying that $\bigcup_{t \geq 0} f^t(q((0,1] \times \mathbb{S}^1)) \subset q((0,1] \times \mathbb{S}^1)$. The image of $q$ is called a collar of attraction of $\gamma$. 
 We start from the classification of limit sets in (a generalisation of) Poincar\'e-Bendixson theorem. 
\begin{thm}[Theorem 2.6.1 of \cite{Nikolaev-Zhuzhoma}]\label{Poincare-Bendixson}
Each $\omega$-limit set of a $\Ci ^1$ flow $F$ with finitely many singular points on a compact surface is one and only one of the following four types: a singular point, a periodic orbit, an  attracting circuit, or a quasi-minimal set.
\end{thm}
It is obvious that an analogous statement for $\alpha$-limit sets also holds due to Theorem 2.6.1 of \cite{Nikolaev-Zhuzhoma}. 
We will show Lemma \ref{main-thm02} according to the  four cases of $\omega (x)$ (and $\alpha (x)$) in  Theorem \ref{Poincare-Bendixson}. 
For convenience, we say that  the limit in \eqref{eq:hb2} for a continuous function $\varphi :S\to \R$ with $B^{\mathcal F}_r(x)$ replaced by \[ B^+_r(x) = \{ y\in O_+(x) \mid d(x,y) <r\} \] is the \emph{length average of $\varphi$ along the forward orbit of $x$}, where $d$ is the distance on $O_+(x)$.

As mentioned in Remark \ref{rmk:1}, we need to note that the time average of a continuous function  along an orbit does not coincide with the length average of the function along the orbit in general. 

We say that a subset $\Sigma$ of $S$ is  a  \emph{cross-section}  of $F$ if $\Sigma$ is  either a  (closed or open) segment  or a circle such that $F$ is transverse to $\Sigma$ and the \emph{first return time} $T_x$ of $x\in \Sigma$ to $\Sigma$ (i.e. the positive  number $t$ such that 
$ f^t(x) \in \Sigma$ and  $ f^s(x) \not\in \Sigma$ for all $0<s<t$) is well defined and finite. 
(Notice that our definition is slightly different with the standard one, cf.~\cite{Smale1967}.) 
Let $\gamma_x=\{ f^t(x) \mid 0\leq t\leq T_x \}$ and $\vert \gamma _x\vert $ the length of $\gamma _x$ for $x\in \Sigma$. 
The following elementary lemma will be used repeatedly.

\begin{lem}\label{lem:boundedlength}
Let $\Sigma$ be a cross-section of  a $\Ci ^1$ flow $F$  on a compact surface $S$. 
Assume that  $\tilde x\in \Sigma $  satisfies that $f^{t}(\tilde x)\not\in \partial \Sigma$ for all $0<t\leq T_{\tilde x}$, where $\partial \Sigma$ is the boundary of $\Sigma$.  
Then, 
$x\mapsto \vert \gamma _{x}\vert $ is continuous at $\tilde x$. 
\end{lem}
\begin{proof}
Take  a real number $\epsilon >0$. 
Let $A$ be the vector field generating $F$, and
 $\ell _A(x) = \vert A(x) \vert _x$ for each $x\in S$, where  $\vert v \vert _{x}$ is the length of $v\in T_xM$ with respect to the Riemannian metric  of $S$  at $x$.
Let $\epsilon _1$ be a positive number smaller than 
$
\epsilon (2 T_{\tilde x} ) ^{-1}
$
and $\Vert \ell _A\Vert _{\infty}=\sup _{x\in S} \vert \ell _A(x) \vert $.
For each $\rho \in (0,T_{\tilde x}) $ and each positive numbers $\delta $,
we consider a \emph{flow box} $B_{\rho ,\delta }$ given by $B_{\rho ,\delta} =\{ f^t(x) \mid t\in [0 ,\rho] , x\in U_\delta \}$, where $U_\delta  =\{  x\in \Sigma \mid d_{\Sigma} (\tilde x, x)<\delta \}$.
Let $\rho$  and $\delta $ be sufficiently small positive  numbers such that
$
\rho < \min \{ T_{\tilde x} ,\epsilon (4\Vert \ell _A \Vert _{\infty} ) ^{-1} \}  ,
$
 that $\frac{ T_{\tilde x}}{\rho}$ is not an integer, and that 
 $
 \vert \ell _A(y) -\ell _A(\tilde y) 
 \vert \leq \epsilon _1
$
whenever $y$ and $\tilde y$ are  in $f^{j\rho} (B_{\rho ,\delta })$ with some $0\leq j\leq N:=\lfloor \frac{T_{\tilde x}}{\rho } \rfloor$ (note that $\ell _A$ is uniformly continuous).

Since  $\frac{ T_{\tilde x}}{\rho}$ is not an integer, we have $ N\rho < T_{\tilde x} < (N+1) \rho$.
Furthermore, due to the hypothesis, we can take $\delta >0$ such  that each flow segment of $f^{N\rho} (B_{\rho, \delta})$ intersects  $\Sigma$ (that is, $\{ f^t (x) \mid t\in [N\rho ,(N+1)\rho ] \}$ intersects $\Sigma$ for each $x\in U_ \delta$).

Let $\gamma _{ x, j} := \{ f^t ( x) \mid t\in [j\rho , (j+1) \rho] \}$ for each $x\in \Sigma$ and $0\leq j\leq N$. Then, it is straightforward to see that
\[
\left\vert \vert \gamma _{\tilde x, j} \vert  -\vert \gamma _{x,j} \vert   \right\vert \leq \rho \epsilon _1
\]
for any $ x\in U_\delta$ and $0\leq j\leq N$. 
Consequently, it holds that for each $ x\in U_\delta$,
\begin{align*}
\left\vert \vert \gamma _{\tilde x} \vert  - \vert \gamma _{ x} \vert  
\right\vert 
&\leq N\rho \epsilon _1
+ 
\vert \gamma _{\tilde x, N} \vert +\vert \gamma _{x, N} \vert \\
&\leq T_{\tilde x} \epsilon _1
+ 2\rho \Vert \ell _A \Vert _{\infty} <\epsilon .
\end{align*}
Since  $\epsilon$ is arbitrary, this completes the proof.
\end{proof}
When $\omega (x) $ is a singular point, it is straightforward to see that  the length average of any continuous function along the forward orbit of $x$ exists. 
(Note that the length of each connected component of the intersection of any orbit and a small neighbourhood of any non-degenerate 
 singularity  is finite; compare with Example \ref{ex2}.)
On the other hand, we need to work a little harder even in the periodic orbit case. 
\begin{prop}\label{IETa}
Let $x$ be a point in $S$ whose $\omega$-limit set is a periodic orbit or an attracting circuit $\gamma$. 
Then the length average of any continuous function $\varphi :S\to \mathbb R$ along the forward orbit of $x$ exists.
\end{prop}
\begin{proof}
We assume that $x\not \in \gamma$ because we can immediately get the conclusion when $x\in \gamma$. 
Let $\Sigma_+$ be a closed segment which is transverse to $\gamma$ at $s_0$ in the boundary of $\Sigma_+$. Let $t_n$ be the $n$-th hitting time of $x$ to $\Sigma _+$ with $n\geq 1$ (i.e.~$0\leq t_n<t_{n+1}$ and  $f^{t}(x) \in \Sigma _+$ if and only if $t\in \{ t_1,t_2,\ldots \}$), and 
$x_n= f^{t_n}(x)$. 
Then, by the assumption $\gamma =\omega (x)$ and the continuity of the vector field  generating $F$, one can find $n_0\geq 1$ such that 
\begin{equation}\label{eq:0129a}
d_{\Sigma _+} (x_{n+1},s_0) <d_{\Sigma _+} (x_n, s_0) \quad  \text{for each $n\geq n_0$.}
\end{equation} 

Let $\Sigma _+(a,b)\subset \Sigma _+$ be the open segment connecting $a$ and $b$. 
One can find $n_1 \geq n_0$ such that the $\omega$-limit set of each point in  $\Sigma _+(x_{n_1}, s_0)$ does not include singular points, 
 because otherwise one can find infinitely many singular points, which leads to  contradiction due to the non-existence of degenerate singularities (see comments above Definition \ref{def:HBF}). 
 Similarly, one can find $n_2 \geq n_1$ such that   the $\omega$-limit set of each point in  $\Sigma _+(x_{n_2}, s_0)$ includes no periodic orbits, 
 because otherwise one can find infinitely many periodic orbits arbitrary close to $\gamma$, which contradicts to the continuity of the vector field  generating $F$ and $\omega(x) = \gamma$. 
Moreover, by the attracting property of $\gamma$, 
one can find $n_3 \geq n_2$ such that the $\omega$-limit set of each point in $\Sigma_+(x_{n_3}, s_0)$ is contained in a collar of attraction of $\gamma$. 
Since no quasi-minimal sets are contained in annuli, by virtue of Theorem \ref{Poincare-Bendixson}, 
the $\omega$-limit set of each point in $\Sigma_+(x_{n_3}, s_0)$ is $\gamma$ and so   
 the first return time $T_{\tilde x}$ of each point $\tilde x$ in $\Sigma := \Sigma _+ (x_{n_3}, s_0)$ to $\Sigma $ 
 is well defined and finite. 
Furthermore, it follows from \eqref{eq:0129a} that  
\begin{equation}\label{eq:0129b}
d_{\Sigma _+} (P(\tilde x),s_0) <d_{\Sigma _+} (\tilde x, s_0) \quad  \text{for each $\tilde x\in \Sigma$},
\end{equation}
 where $P: \Sigma \to \Sigma$ is the forward Poincar\'e map on $\Sigma$. 

%
%
By Lemma \ref{lem:boundedlength},
 $\tilde x\mapsto  \vert \gamma _{\tilde x}\vert $ is continuous on the cross-section $\Sigma$, and thus,  a function $\tilde \varphi :\Sigma  \to \R$  given by
\begin{equation}\label{eq:timeonefunction}
\tilde \varphi (\tilde x) = \int _{\gamma _{\tilde x}} \varphi (y) dy \quad (\tilde x\in \Sigma)
\end{equation}
 is also  continuous for  any continuous function $\varphi :S\to \R$, where $ \gamma _{\tilde x}=\{ f^t(\tilde x) \mid 0\leq t\leq T_{\tilde x} \}$. 
By \eqref{eq:0129b}, for any continuous function $\tilde \varphi _1:\Sigma _+ \to \R$ and $\tilde x\in\Sigma$, 
we have
\begin{equation}\label{eq:contraction}
\lim _{N\to \infty} \frac{1}{N} \sum _{j=0}^{N-1} \tilde \varphi _1(P^j (\tilde x)) 
= \lim _{\tilde y\to s_0} \tilde \varphi _1(\tilde y),
\end{equation}
in particular, the time averages of $\tilde \varphi _1$ at $\tilde x$ 
exists. 
On the other hand,
for all large $r$, denoting by $N_r$ the maximal integer such that $P^{j}(x_{n_3}) \in B^+_r(x)$ for each $0\leq j\leq N_r-1$, 
 we can rewrite the integral  of $\varphi$ along the forward orbit of $x$ by
 \begin{equation}\label{eq:rew}
  \int _{B^+_r(x)} \varphi (y)dy = \int _{\tilde \gamma _0} \varphi (y) dy +\sum _{j=0}^{N_r-1} \tilde \varphi (P^j(x_{n_3})) +\int _{\tilde \gamma _1} \varphi (y) dy, 
\end{equation}
where  $\tilde \gamma _0 =\{ f^t(x) \mid 0\leq t\leq t_{n_3}\}$ and  $\tilde \gamma _1 =B_r^+(x) - \tilde \gamma _0 - \bigcup _{j=0}^{N_r-1} \gamma _{P^j(x_{n_3})}$, whose lengths are bounded by a constant  independently of $r$.

Note that $\lim _{\tilde x\to s_0} \vert \gamma _{\tilde x}\vert \in (0,\infty )$ because the vector field generating $F$  is continuous  and $\vert \gamma \vert \in (0, \infty)$, and so we have $\lim _{\tilde x\to s_0} \vert \tilde \varphi ( \tilde x) \vert <\infty$ due to the form of $\tilde \varphi$ in \eqref{eq:timeonefunction}.
Applying  \eqref{eq:rew} to $\varphi \equiv 1$ and  \eqref{eq:contraction} to $\tilde \varphi _1(\tilde x) = \vert \gamma_{\tilde x} \vert$, we get that
\begin{equation}\label{eq:boundedlength}
\lim _{r\to \infty} \frac{  \vert B^+_r(x)\vert}{N_r} = \lim _{N\to \infty}\frac{1}{N} \sum _{j=0}^{N-1}\vert \gamma _{P^j(x_{n_3})}\vert =\lim _{\tilde x\to s_0}\vert \gamma _{\tilde x} \vert .
\end{equation}
Hence, it follows from   \eqref{eq:rew} that 
\[
 \lim _{r\to \infty} \frac{1}{\vert B^+_r(x)\vert} \int _{B^+_r(x)} \varphi (y)dy = \left(\lim _{\tilde x\to s_0}\vert \gamma_{\tilde x}\vert \right)^{-1} \lim _{N\to \infty} \frac{1}{N} \sum _{j=0}^{N-1} \tilde \varphi (P^j(x_{n_3})),
\]
if the limits exist. 
By using   \eqref{eq:contraction} again, we can see that the limit in the right-hand side exists and coincides with $\lim _{\tilde x\to s_0} \tilde \varphi (\tilde x) /\vert \gamma _{\tilde x}\vert $. 
This completes the proof.
\end{proof}

\begin{remark}\label{rmk:Bowen}
The Bowen flow has an open  set $D$ surrounded by an attracting circuit consisting of two saddle singularities $p_1$ and $p_2$ and two heteroclinic orbits $\gamma _1$ and $\gamma _2 $ connecting theses singularities  such that the time average of some continuous function  does not exist at $x$ and $\omega (x) $ is the attracting circuit $\gamma := \{p_1\} \sqcup \{p_2\} \sqcup \gamma _1 \sqcup \gamma _2$ for every point $x$ in $D$ except the source $ \hat{p}$ of the flow 
   (\cite{Takens1994}), see Figure \ref{fig0}. 
\begin{figure}[h]\begin{center}
\includegraphics[scale=0.45]{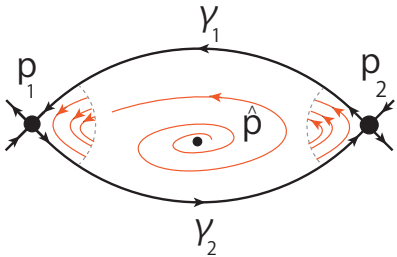}
\end{center} 
\caption{Bowen's flow}\label{fig0}
\end{figure} 
On the other hand, length averages exist at every point in $D$ for the singular foliation generated by the Bowen's flow due to Proposition \ref{IETa}.
The key  is that the \emph{length} of $\tilde \gamma _1$ in the proof of Proposition \ref{IETa}  is bounded by a constant independently of  $r$, while the \emph{time}  to pass $\tilde \gamma _1$, that is $t-T_{n_3} - \sum _{j=0}^{N_r-1} T_{P^j(x_{n_3})}$, becomes larger as $r$ (or $t$) increases since $f^t(x)$ is  ``trapped'' in a neighbourhood of $p_1$ and $p_2$ longer as the ``depth'' $N_r$ increases.
\end{remark}

Finally, we will  consider the existence of length averages in the case when the $\omega$-limit
set is a quasi-minimal set. 
 It follows from the structure  theorem of Gutierrez \cite{Gutierrez1986} (see also Lemma 3.9 of \cite{Gutierrez1986} and Theorem 2.5.1 of \cite{Nikolaev-Zhuzhoma}) 
  that for any $\Ci ^1$ flow $F$, 
there are at most  finitely many quasi-minimal sets, and that for each quasi-minimal set $\overline{O}$ 
(i.e.~the closure of a non-closed and recurrent orbit $O$), one can find  a circle $\Sigma$ which is transverse to $O$ such that the forward Poincar\'e 
map $P : \Sigma \to  \Sigma$ is well-defined and topologically semi-conjugate to a minimal
interval exchange transformation $E: \mathbb S^1\to \mathbb S^1$ on the circle $\mathbb S^1$
(recall that an interval exchange transformation is said to be \emph{minimal} if every orbit is dense; 
refer to \cite{Viana2006} for definition and basic properties  of interval exchange transformations). 
Moreover, the semi-conjugacy $h: \Sigma \to \mathbb S^1$ is constructed as the quotient map induced by pairwise disjoint (possibly empty) closed intervals $\{ I_n\}_{n\geq 1}$ of $\Sigma$, that is, $h(I_n)$ is a point set for each $n\geq 1$, the restriction of $h$ on $\Sigma - \cup _{n\geq 1} I_n$ is injective and $h\circ P = E\circ h$. 
We call  $E$  the  associated interval exchange transformation of the quasi-minimal set $\overline{O}$. 
\begin{dfn}\label{dfn:ue}
A quasi-minimal set of a  $\Ci ^1$ flow $F$ is said to be \emph{uniquely ergodic} if the  associated interval exchange transformation  is   uniquely ergodic  (i.e.~there is a unique invariant probability measure of the transformation). 
Furthermore, a quasi-minimal set of a  $\Ci ^1$ flow $F$ is said to be \emph{locally dense} 
 if $h$ is 
 injective. 
When every quasi-minimal set of $F$ is uniquely ergodic or locally dense, we simply say that the union of quasi-minimal set of $F$ is 
uniquely ergodic or locally dense, respectively.  
\end{dfn}

\begin{remark}\label{rmk:r121}
There is an example of a minimal but non-uniquely ergodic  interval exchange transformation 
(\cite{Keynes-Newton1976}).
On the other hand, it is rather typical for interval exchange transformations to be uniquely ergodic: 
Recall that an interval exchange transformation is determined by the number of intervals $N\geq 2$, a   length vector $(\lambda _1, \lambda _2, \ldots ,\lambda _N) \in \mathbb R^N$ and a permutation  of $\{ 1,2, \ldots , N\}$. 
%
It is a famous deep result 
 (\cites{Masur1982, Veech1982}) that for each $N$, $\pi$ and Lebesgue almost every $\lambda =(\lambda _1, \lambda _2, \ldots ,\lambda _N) $,  the interval exchange transformation given  by $N$, $\lambda $ and $\pi$ is uniquely ergodic. 

Furthermore, it is known that any interval exchange transformation with $N\leq 3$ is uniquely ergodic (\cite{Keane1975}). We show that $N\leq 2$ always holds for the associated interval exchange transformations of a $\Ci^1$ flow $F$ on a compact surface $S$ if  the orientable genus of $S$ is less than two or the non-orientable genus of $S$ is less than four.  It is known that the total number of quasi-minimal sets of $\mathcal F$ cannot exceed $g$ if $S$ is an orientable surface of genus $g$ \cite{Mayer1943},
and $\frac{h-1}{2}$
if $S$ is a non-orientable surface of genus $h$ \cite{Markley1970}.
So,  $S$ has no quasi-minimal set when the  orientable genus of $S$ is zero or the non-orientable genus  of $S$ is less than three. 
Moreover, since each boundary component of a foliated surface  with no degenerate singularities is either transverse to the foliation or a union of leaves, 
we can assume that $S$  is a closed surface. 
Therefore, it is a direct consequence of Denjoy-Siegel Theorem that $N\leq 2$ when the orientable genus of $S$ is one. 

We consider the number of intervals $N$ for a closed surface $S$ with non-orientable genus  three. 
By classification theorem of closed surfaces, the Euler characteristic of $S$ is $2 - 3 = -1$. 
So, by cutting $S$ along a circle $\Sigma$ which is transverse to a quasi-minimal set of $F$, we obtain a projective plane with two punctures
(note that its  Euler characteristic  is also $1 - 2 = -1$).
The resulting surface is drawn on the left-side of 
 Figure \ref{fig:proje}, and 
 two small circles (with $+$ and $-$) on the left-hand side correspond to $\Sigma$. 
 Now we consider two (parts of) trajectories  $a$ and $b$ which start from the circle with $+$ and land on the circle with $-$ (i.e.~transversely return to $\Sigma$). 
 When $N\geq 2$, we can take $a$ and $b$ as they are not parallel as depicted in  Figure \ref{fig:proje}. 
 Here by parallel we mean that there is a trivial flow
box whose transverse boundaries are $a$ and $b$ and whose tangential
boundaries are the circles with $+$ and $-$.
 However, another trajectory $c$ 
 has to be parallel to $a$ or $b$ 
  (see the right-hand side of Figure \ref{fig:proje}), and thus, we immediately get $N\leq 2$. 

\begin{figure}[h]
\begin{center}
\includegraphics[scale=0.25]{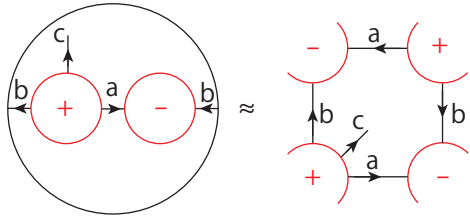}
\end{center} 
\caption{Trajectories on a projective plane with two punctures}
\label{fig:proje}
\end{figure}
%
\end{remark}

\begin{prop}\label{IET2}
There is a Lebesgue zero measure set $A$ such that if $x$ is outside of $A$ and the $\omega$-limit set of $x$ is a locally dense quasi-minimal set, then  the length average of any continuous function 
  along the forward orbit of $x$ exists. 
\end{prop}

\begin{proof}
Let $P$ be the Poincar\'e map on   a circle  $\Sigma$ which transversely intersects a quasi-minimal set $\overline{O}$. 
Since any interval exchange transformation preserves a Lebesgue measure, by the hypothesis,  there is a zero Lebesgue  measure set  $A_0 \subset \Sigma$ such that the time average of $P$ for any continuous function on $\Sigma$ exists at every $x$ outside of $ A_0$. 
Let $O_-( A_0) =  \bigcup _{t\leq 0} f^t (A_0) $, then $O_-( A_0)$ be a Lebesgue zero measure set of $S$
 (note that $ O_-( A_0)$ can be decomposed into countably many  zero measure sets $ \bigcup _{-n -1 < t\leq -n } f^t (A_0)$ with $n\geq 0$). 
Let $y$ be a point outside of $O_-( A_0)$ such that the $\omega $-limit set  of $y$ is $\overline{O}$. 
Then,  one can find a real number $t_1\geq 0$ such that $f^{t_1}(y)\in \Sigma$ because $\omega (y) =\overline{O}$, and 
 it follows from $y\not \in O_-( A_0)$ that  $f^{t_1}(y) \not\in A_0$. 
 Therefore,  the length average of any continuous function exists along the forward orbit of $y$ 
by Lemma \ref{lem:boundedlength} (see the proof of Proposition \ref{IETa}). 
Since there are at most finitely many quasi-minimal sets (cf.~\cites{Gutierrez1986, Nikolaev-Zhuzhoma}), this completes the proof. 
\end{proof}

\begin{prop}\label{IET2r2}
Let   $x$ be a point whose $\omega $-limit set is a uniquely ergodic quasi-minimal set. 
Then,  the length average of any continuous function 
  along the forward orbit of $x$ exists. 
\end{prop}

\begin{proof}
The proof is basically same as the proof of Proposition \ref{IET2} (recall that if a map is uniquely ergodic, then time averages exist everywhere for the map \cite{KH95}). 
The only difference is  that 
  the semi-conjugacy $h$ between the associated Poincar\'e map $P:\Sigma \to \Sigma$ and the associated  interval exchange transformation $E: \mathbb S^1\to \mathbb S^1$ of  the quasi-minimal set of $x$ may  be not injective, that is, there is   pairwise disjoint nonempty closed intervals $\{ I_n\}_{n\geq 1}$ of $\Sigma$ such that $h(I_n)$ is a point set for each $n\geq 1$. 
On the other hand, it can be   managed due to  the unique ergodicity of $E$: Let $\mu$ be an invariant measure of $P$, then the pushforward measure $h_*\mu$ of $\mu$ by $h$ (given by $h_*\mu (A) =\mu (h^{-1} A)$ for each Borel set $A\subset \mathbb S^1$) is an invariant measure of $E$. 
Moreover, since $E$ is a uniquely ergodic interval exchange transformation, $h_*\mu$ is the Lebesgue measure of $\mathbb S^1$, and thus we have $\mu (I_n) \leq h_*\mu (h(I_n))=0$ for each $n\geq 1$ 
because $h(I_n)$ is a point set and $I_n \subset h^{-1} \circ h(I_n)$. 
This immediately concludes that $P: \Sigma \to \Sigma$ is uniquely ergodic, and we complete the proof by repeating the argument in the proof of Proposition \ref{IET2}. 
%
%
%
%
\end{proof}

Theorem  \ref{main-thm02} immediately follows from Theorem \ref{Poincare-Bendixson}, Propositions \ref{IETa}, \ref{IET2} and \ref{IET2r2}.

\subsection{Proof of main theorems}\label{subsection-2}

We will reduce existence of length averages for a  $\Ci^1$ singular foliation without 
degenerate singularities in Theorem \ref{main-thm} to existence of length averages for a singular foliation generated by a $\Ci ^1$ flow in Theorem \ref{main-thm02}. 
 We say that a singular foliation $\mathcal F$ on a compact surface $S$ is \emph{orientable} if the restriction of $\mathcal F$ on $S - \mathrm{Sing}(\mathcal F)$ is orientable. 
We need the following observation from \cite{HH1986A}*{Remark 2.3.2}.

\begin{lem}\label{FF}
Any one-dimensional $\Ci^1$ singular foliation on a closed manifold $(\mathrm{i.e.}$~a compact manifold without boundaries$)$ is orientable if and only if it can be generated by a $\Ci^1$ flow. 
\end{lem}



Fix  a compact surface $S$ and a codimension one $\Ci ^1$  singular foliation $\mathcal{F}$ without 
 degenerate singularities on  $S$.
To define unique ergodicity  for $ \mathcal F$ (in Definition \ref{dfn:ue2}), 
 we need construct a codimension one $\Ci ^1$ orientable singular foliation $\hat{\mathcal F}$ on an orientable compact surface  $\hat S$ without boundary, together with a continuous surjection $p: \hat S \to S$ such that $\hat{\mathcal F}$ is the lifted singular foliation of $\mathcal F$ by $p$. 

Firstly, if  $S$ is orientable, simply set $\hat S_1=S$ and $\hat{\mathcal F}_1=\mathcal F$, together with the identity map $p_0: \hat S_1 \to S$.
In the case when $S$ is not orientable, 
we consider  the usual \emph{orientation double covering} $\hat S_1$  of $S$ (cf.~\cite{D1995}) and    its induced foliation $\hat{\mathcal F} _1$ by  the canonical projection $p_0:\hat S_1 \to S$. 
(That is, we construct  $\hat S_1$ as the set  of pairs $(y, o)$, where $y$ is a point in $S$ and $o\in \{+,-\}$ is the orientation of $S$, and $p_0$ is given by $p_0(y,o)=y$.) 

Secondly, when $\hat S_1$ has no boundaries, simply set $\hat S_2=\hat S_1$ and $\hat{\mathcal F}_2=\hat{\mathcal F} _1$, together with the identity map $p_1: \hat S_2 \to \hat S_1$.
When $\hat S_1$ has a  boundary, let $\hat S_2$ be the double of $\hat S _1$ (i.e., $\hat S _2 = \hat S_1 \times \{ 0,1\} / \sim$, where $(x, 0)\sim (x,1)$ for all $x \in \partial \hat S_1$). Note that $\hat S_2$ has no  boundaries. 
See Figure \ref{pic01}.
We let $p_1: \hat S_2 \to \hat S_1$ be  the canonical projection, given by $p_1(y,j)=y$ for $y \in \hat S_1 - \partial \hat S_1$ and $j=0,1$,  and $p_1([y,0])=y$ for $y\in \partial \hat S_1$.
Furthermore, we let  $\hat{\mathcal F} _2$ be the lifted foliation of $\hat{\mathcal F}_1$ on $\hat S_2$ 
 (denoted by $\hat{\mathcal F} _1 \sqcup _{\partial } -\hat{ \mathcal F}_1$ in Figure \ref{pic01}).

\begin{figure}[h]
\begin{center}
\includegraphics[scale=0.25]{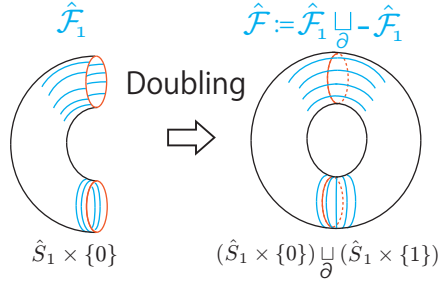}
\end{center} 
\caption{The lifted foliation $\hat \F$ (or $\hat \F _2$) of $\hat{\mathcal{F}}_1$  on the double $\hat S$ of $\hat S_1$}
\label{pic01}
\end{figure}

Thirdly, when  $\hat{\mathcal F}_2$ is orientable, simply set $\hat S=\hat S_2$ and $\hat{\mathcal F}=\hat{\mathcal F}_2$, together with the identity map $p_2: \hat S \to \hat S_2$.
In the case when $\hat{\mathcal F}_2$ is not orientable, 
we consider  the  \emph{tangent orientation double covering} $(\hat S^\prime , \hat{\mathcal F}^\prime )$  of $(\hat S_2 -\mathrm{Sing} (\hat{\mathcal F} _2), \hat{\mathcal F} _2 -\hat{\mathcal S}_2)  $ (cf.~2.3.5, p.16 of \cite{HH1986A}) with  the canonical projection $p_2^\prime :\hat S^\prime  \to \hat S_2$ (i.e.~$p_2^\prime (y,o)=y$ for  $y\in \hat S_2-\mathrm{Sing} (\hat{\mathcal F}_2)$ and the orientation $o\in \{ +,-\}$ of $\hat{\mathcal F}_2$ at $y$), where $\hat{\mathcal S}_2$ is the set of singular leaves of $\hat{\mathcal F}_2$. 
Let $\hat S= \hat S^\prime \sqcup (\mathrm{Sing} (\hat{\mathcal F} _2) \times \{ 0,1\})$, and define $p_2: \hat S\to \hat S_2$ by $p_2(x)=p_2^\prime (x)$ for $x\in \hat S^\prime $ and  $p_2((s,j))=s$ for $s\in\mathrm{Sing}(\hat{\mathcal F}_2)$ and $j\in \{0,1\}$. 
We inductively  define   a metric $d_{\hat S}$ on $\hat S$ as follows: 
Let $d_{\hat S}(x,y)=d_{\hat S^\prime}(x,y)$ for each $x, y\in \hat S^\prime $, where $d_{\hat S^\prime}$ is the induced metric of $d_{\hat S_2}$ by the double covering $p_2^\prime$. 
Since  $\mathrm{Sing} (\hat{\mathcal F}_2)\times \{0,1\}$ is a finite set (see comments above Definition \ref{def:HBF}), we can denote it by 
$ \{ s_1, s_2,\ldots ,s_N\}$. 
Assume that  $d_{\hat S}$ is defined on $\hat S^\prime \sqcup \{ s_1, s_2, \ldots ,s_k\}$ with some $0\leq k<N$ (we let $ \{ s_1,  \ldots ,s_k\}=\emptyset$ if $k=0$, for convenience). 
Then, it is straightforward to see that the restriction of $\hat{\mathcal F}_2$ on a small ball $U$ centred at $p_2(s_{k+1})$
is orientable (because 
each point in $\mathrm{Sing} (\hat{\mathcal F} _2)$ is a sink, a source, a center or a saddle of a closed manifold $\hat S_2$),  
and 
thus, $p_2^{-1}(U-\{ p_2(s_{k+1})\} )$ consists of two connected components, (at least) one of which has a positive distance from $\{s_1, \ldots ,s_k\}$. 
Denote the connected component  by $A_{k+1}$, so that we have $d_{\hat S}(A_{k+1} , \{ s_1, \ldots , s_k\} )>0$. 
For each $x\in \hat S^\prime \sqcup \{s_1, s_2, \ldots ,s_k\}$, let  $d_{\hat S}(s_{k+1}, x):= \lim _{n\to \infty}d_{\hat S }(y_n , x)$ with some sequence $\{y _n\}_{n\geq 1}\subset A_{k+1}$ such that  $\lim _{n\to \infty}d_{\hat S_2}( p_2(s_{k+1}) ,p_2 (y_n))=0$. 
By construction, it is not difficult to see that $p_2: \hat S\to \hat S_2$ is a double covering, and that 
 the smooth structure of $\hat S_2$ can induce a smooth structure of $\hat S$ by  
$p_2$. 
 Since the restriction of $\hat{\mathcal F}$  is orientable both on $p_2^{-1}(\hat S_2-\mathrm{Sing} (\hat{\mathcal F}_2))$ (due to the orientability of $\hat{\mathcal F}^\prime$) and on a small  neighbourhood of $p_2^{-1}(\mathrm{Sing}(\hat{\mathcal F} _2))$, 
   $\hat{\mathcal F}$ is an orientable foliation on an orientable closed surface $\hat S$. 

We call the resulting foliation $ \hat{\mathcal F}$
the \emph{associated foliation} 
  of $\mathcal F$.
Note that  $ \hat{\mathcal F}$ can be   generated by a $\Ci ^1$ flow due to Lemma \ref{FF}. 
We need the following definition. 

\begin{dfn}\label{dfn:ue2}
We say that the union of quasi-minimal sets of a codimension one $\Ci ^1$  singular foliation $ \mathcal F$ on a compact surface $S$ is \emph{uniquely ergodic} or \emph{locally dense} if the union of quasi-minimal sets of the flow generating the associated foliation $\hat{\mathcal F}$ is  uniquely ergodic or locally dense, respectively (see Definition \ref{dfn:ue}). 
\end{dfn}

We are now ready to prove Theorem \ref{main-thm} and Corollary \ref{main-thm01b}. 

\begin{proof}[Proof of Theorem \ref{main-thm}]

Fix  a codimension one $\Ci ^1 $ singular foliation $\mathcal F$ on a compact surface $S$. 
Let $(\hat S, \hat{\mathcal F})$ be the associated foliated manifold of $(S, \mathcal F)$ given  above. 
We will show that if  length averages exist at $x  \in  \hat S$ for  $\hat{\mathcal{F}}$, then length averages also exist at $p(x)$  for $\mathcal F$, which immediately completes the proof by Theorem \ref{main-thm02} and Lemma \ref{FF}.

First we will show that if length averages exist at $x\in \hat S_1$ for $\hat \F_1$, 
then length averages also exist at $ p_0(x) \in S$  for  $\mathcal{F}$. 
It is trivial in the case $(\hat S_1 ,\hat \F _1)=(S, \F )$, so we consider the other case. 
Assume that length averages exist at $x  \in \hat S_1$  for $\hat \F _1$. 
We also assume that $\hat{\mathcal{F}}_1 (p_0(x))$ is not compact because otherwise the claim obviously holds.
Let $\varphi :S \to \R$ be  a continuous function.
Then,  $\hat{\varphi}:= \varphi \circ p_0$ is continuous, and thus   the length average of $\hat{\varphi}$ along the leaf of $\hat{ \mathcal F}_1$ through $x$, given by
\begin{equation}\label{eq:1108a}
\lim _{r\to \infty} \frac{1}{\vert B_r^{\hat{ \mathcal F}_1} ( x) \vert } \int _{B_r^{\hat{ \mathcal F}_1} (x)} \hat{\varphi} (\tilde x)d\tilde x,
 \end{equation}
 exists because of the assumption for length averages on $x$. 
On the other hand, it is straightforward to see that  $\vert \det Dp_0 \vert \equiv 1$ and 
$p_0(B_r^{\hat{ \mathcal F}_1} ( x)) =B_r^{\mathcal F} (p_0(x))$. 
(Note that, although $p_0$ is not injective, $p_0 \vert _L$ is injective for all non-compact $L\in \hat{\mathcal F} _1$
since $p_0(L)$ is a leaf.)
Therefore, by applying  the change of variables formula, we have that
\begin{equation}\label{eq:1108b}
\frac{1}{\vert B_r^{\mathcal F}  (p_0(x)) \vert } \int _{B_r^{\mathcal F} (p_0(x))} \varphi  (\tilde y)d\tilde y = \frac{1}{\vert B_r^{\hat{ \mathcal F}_1} ( x) \vert } \int _{B_r^{\hat{ \mathcal F}_1} (x)} \hat{\varphi} (\tilde x)d\tilde x .
\end{equation}
Since $\varphi$ is arbitrary, it follows from \eqref{eq:1108a} and \eqref{eq:1108b} that length averages exist at  $p_0(x)$  for $\F$.

Next we will show that if length averages exist at $x\in \hat S_2$ for $\hat{\mathcal F}_2$, 
then  length averages also exist at $p_1(x) \in \hat S_1$  for  $\hat{\mathcal{F}} _1$. 
Only  the case when $\hat S_2$ is the double of  $\hat S_1$  is considered because the other case is trivial. 
Let $x\in \hat S_2$ be a point at which length averages exist. 
If  $\hat{\mathcal F}_1(p_1(x))$ does not have an intersection with $\partial  \hat S_1$, then it is straightforward to see that for any continuous function $\varphi$ on $\hat S_2$,  the length average of   $\varphi$ along $\hat{\mathcal F}_2(x)$ coincides with the length average of a continuous function $\hat{\varphi } = \varphi \circ p_1$ along  $\hat{\mathcal F}_1(p_1(x))$
%
 (see the argument in the previous paragraph), so length averages at $p_1(x)$ exist. 
Moreover, if $\hat{\mathcal F}_1(p_1(x))$ is included in $\partial  \hat S_1$, then the length of $\mathcal F_1(p_1(x))$ is finite and  length averages at $p_1(x)$  exist. 
Therefore, 
we consider the case when $\hat{\mathcal F}_1(p_1(x)) \cap \partial \hat S_1 \neq \emptyset$ and $\hat{\mathcal F}_1(p_1(x)) \not \subset \partial \hat S_1$. 
If $\hat{\mathcal F}_1(p_1(x)) $ is transverse to $ \partial \hat S_1 $ with respect to both direction, then $\hat{\mathcal F}_1(p_1(x)) $  is (diffeomorphic to) a closed interval, and so length averages at $p_1(x)$ exist. 
If $\hat{\mathcal F}_1(p_1(x)) $ is transverse to $ \partial \hat S_1 $ at exactly one point $p_1(y)$ with $y\in \partial \hat S_2$, then by repeating the argument in the previous paragraph, one can check that 
   for any continuous function $\varphi$ on $\hat S_2$,  the length average of   $\varphi$ along $\hat{\mathcal F}_2(x)=\hat{\mathcal F}_2(y)$ coincides with  double of  the length average of a continuous function $\hat{\varphi } = \varphi \circ p_1$ along  $\hat{\mathcal F}_1(p_1(x)) = \hat{\mathcal F}_1(p_1(y))$, so length averages exist at $p_1(x)$.

Finally we need to  show that if length averages exist at $x\in \hat S$ for $\hat{\mathcal F}$, 
then length averages also exist at $p_2(x) \in \hat S_2$  for  $\hat{\mathcal{F}} _2$. 
However, it is completely analogous to the previous argument for the  double covering $p_0: \hat S_1\to S$ because $p_2: \hat S\to \hat S_2$ is also a double covering. 
This completes the proof. 
\end{proof}


\begin{proof}[Proof of Corollary \ref{main-thm01b}]
Let $S$ be a compact surface with a regular foliation $\mathcal F$. 
Since the Euler characteristic of any manifold with a regular foliation is $0$ by the Poincar\'e-Hopf theorem, $S$ is a torus, an annulus,   a M\"obius band or a Klein bottle.
On the other hand, the total number of quasi-minimal sets of $\mathcal F$ cannot exceed $g$ if $S$ is an orientable surface of genus $g$ \cite{Mayer1943},
and $\frac{h-1}{2}$
if $S$ is a non-orientable surface of genus $h$ \cite{Markley1970}.
So, when $S$ is not a torus, there is no quasi-minimal set of $\mathcal F$. 
When $S$ is a torus, the union of quasi-minimal sets of $\mathcal F$ is uniquely ergodic by Denjoy-Siegel theorem. 
Therefore, we complete the proof of Corollary \ref{main-thm01b} by Theorem  \ref{main-thm}.
\end{proof}

\section{Examples}\label{example}

\begin{example}\label{ex1}
There is a smooth codimension one regular foliation on an unbounded manifold without length averages. 
Let $M=\mathbb R^k \times \mathbb R^{m-k}$ and $\mathcal{F} = \{ \mathbb{R}^k \times \{ t \} \mid t \in \mathbb{R}^{m-k} \}$. 
Let $\varphi: M \to [-1,1]$ be  a continuous function  such that
$\varphi(x,y) = 1$ if 
 $\vert x \vert \in \bigcup _{n \in \mathbb{N}} (10^{2n-1} +1 , 10^{2n} -1 )$  and 
$\varphi(x,y) = -1$ if 
 $\vert x \vert \in \bigcup_{n \in \mathbb{N}} (10^{2n} +1 , 10^{2n+1} -1 )$ with $x\in \mathbb R^k $ and $y\in  \mathbb R^{m-k}$. 
Then, it follows from a straightforward calculation  that 
 the length average  of  $\varphi$ does not exist. 
Moreover, if the above foliated manifold can be isometrically embedded into a foliated $m$-dimensional manifold with $k$-dimensional leaves, then there is a nonempty open subset which consists of leaves  without length averages.

\end{example}

\begin{example}\label{ex2}
There is a smooth codimension one regular foliation on an incomplete manifold 
without length averages. 
We are indebted to Masayuki Asaoka for the construction.
Let $M$ be an open disk in $\mathbb R^2$. Let $D_n \subset M$ ($n\geq 1$) be a square  whose side has a length $(\frac{3}{4})^n$ such that $D_n$ borders on $D_{n+1}$ in the right side and that $D_n$ accumulates to the boundary of $M$. 
Also, define  a piecewise linear  leaf $L$ of $M$ as  $L\cap D_n$ is the Koch curve of level $n$ scaled by $(\frac{3}{4})^n$ (refer to \cites{Koch1904,falconer2004fractal} for the definition of the Koch curve).
See Figure \ref{fig01aa}.
Since the length of the Koch curve of level $n$ is $(\frac{4}{3})^{n}$, the length of $L\cap D_n$ is $1$.
Finally, let  $U$ be a small neighbourhood of $\cup _{n\in \mathbb N}D_n $ such that $L\cap (U - \cup _{n\in \mathbb N}D_n  )$ has  finite length.
Let   $\varphi :M\to \mathbb R$ be a uniformly bounded continuous function such that $\varphi (x)=1$ if $x \in D_n$ ($10^{2m-1}+1\leq n\leq 10^{2m}-1$, $m\geq 1$),  $\varphi (x) =-1$ if $x\in D_n$ ($10^{2m}+1\leq n\leq 10^{2m+1}-1$, $m\geq 1$) and  $\varphi (x) =0$ if $x\not\in U$.
Then, it is easy to check that on $L$, the length average of $\varphi$ (in the sense of piecewise $\Ci ^1$) does not exist.   
Furthermore, one can find  a smooth curve along which the length average of $\varphi$ does not exist by applying a standard argument for mollifiers to the piecewise linear curve $L$.\footnote{  
Let $\ell  \equiv (\ell _1 , \ell _2) : \mathbb R\to M$ be a parametrisation of $L$ with velocity $1$ (except singular points of $L$) such that  $\ell (n)$ is the right endpoint of $L\cap D_n$ for each $n\geq 1$.
Let $\chi :\mathbb R\to [0,1]$ be a $\mathscr C^\infty$, compactly supported,  symmetric function  such that $\int _{\mathbb R} \chi (t) dt=1$. 
For each positive number $\delta$, let $\ell _{\delta } \equiv (\ell _{\delta ,1} , \ell _{\delta ,2}): \mathbb R \to M$ be a smooth function such that   $\ell _{\delta , j }$ ($j=1, 2$) is the convolution of $\ell _j$ and $t\mapsto \delta ^{-1} \chi (\delta ^{-1}t )$.  
Let $\mathcal A_n$ be the   set of singular points of $L$ in $D_n$ and $a_n=\# \mathcal A_n$ for $n\geq 1$. 
Fix $\epsilon \in (0,1)$. 
Then, it is straightforward to  check that for each $n\geq 1$, there is a positive number $\delta (n) $ such that $\ell _{\delta (n )} ([n-1,n])$ is a smooth curve  agreeing with $L\cap D_n$ except the $(\frac{3}{4}) ^{2n} \cdot \frac{\epsilon }{a_n}$-neighbourhood of $\mathcal A_n$. 
Therefore, along a smooth curve $\tilde L$  satisfying  that $\tilde L\cap \left(\cup _{n\in \mathbb N} D_n\right) = \cup _{n\in \mathbb N} \ell _{\delta (n )} ([n-1,n])$ and $\tilde L \cap  (U - \cup _{n\in \mathbb N}D_n  )$ has finite length,  
 the length average of $\varphi$  also does not exist.
%
}

\begin{figure}[h]\label{fig01aa}
\begin{center}
\includegraphics[scale=0.35]{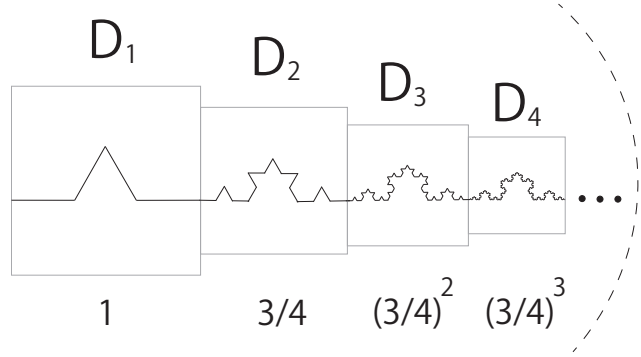}
\end{center} 
\caption{A codimension one regular foliation on an open disk 
without  length averages}
\end{figure}
\end{example}

\begin{example}\label{ex3}

One can find points without length averages for a smooth codimension one foliation with degenerate singularities on a compact surface.
The idea of the construction is basically same as the previous example.
Let $M$ be a compact surface and $L$ be a leaf constructed in the same manner as in the previous example, except that $D_n$ accumulates to a point in $M$. Then,  on the leaf $L$, some length average does not exist   by virtue of the argument in the previous example (so that the accumulation point must be a degenerate singular point by Theorem \ref{main-thm}).
 
\end{example}

\begin{remark}
Examples \ref{ex2} and \ref{ex3} may show that (non-)existence of length averages for foliations is not preserved by some homeomorphisms of foliations (but preserved by any $\Ci ^1$ diffeomorphisms of foliations), while (non-)existence of time averages for flows  is preserved by  topological  conjugacy.
\end{remark}

\begin{example}\label{ex4}
There is a codimension \emph{two} regular foliation on a compact manifold without length averages. 
Let $P$ be a diffeomorphism on a compact manifold $N$ such that there exists a positive Lebesgue measure set $D$ consisting of points without time averages. 
 As we mentioned in Section \ref{section:introduction}, there are several examples of  diffeomorphisms $P$ satisfying the condition for time averages.
 Furthermore, for simplicity, we assume that the backward orbit of each $x\in D$ along $P$ accumulates to a source $\hat{p}$, such as the time-one map of the Bowen flow (pictured in Figure \ref{fig0}, see Remark \ref{rmk:Bowen}).
 Let $\mathcal F$ be the trivial suspension of $P$. That is, $M=N \times [0,1]/\sim$, where $(\tilde x,1)\sim (P(\tilde  x) ,0)$ for $\tilde x\in N$, and each leaf $L$ of $\mathcal F$ is of the form $\left(\bigcup _{n\in \mathbb Z}\{ P^n(\tilde x)\} \times [0,1] \right)/\sim$.
 Fix $\tilde x\in D$ and a continuous function $\tilde \varphi :N \to \mathbb R$ such that the time average of $\tilde \varphi $ along the forward orbit of $\tilde x$ does not exist. 
 Let $\varphi $ be a continuous function on $M$ given by $\varphi (y,s) =\tilde \varphi (y)$ for each $(y,s) \in N \times (0,1)$, so that $\tilde \varphi (y) =\int _{B_{1/2} ^{\mathcal F}((y, \frac{1}{2}))} \varphi (z) dz$.  
 Then,  it is not difficult  to check that the length average of $\varphi $ along the leaf $\mathcal{F}((\tilde x,\frac{1}{2}))$, given    in \eqref{eq:hb2}, equals to  $\left(\lim _{n\to \infty} \frac{1}{n}\sum _{ j=0}^{  n-1} \tilde \varphi (P^j (\tilde x)) + \tilde \varphi (\hat{p}) \right) /2$, which does not exist due to the choice of $\tilde x$ and $\tilde \varphi$.

\end{example}

\begin{remark}\label{rmk:ex4}
Example \ref{ex4} implies a canonical correspondence between time averages for a \emph{suspension} flow and length averages for the foliation generated by the flow. 
This observation may be  helpful when one considers existence problem of length averages for  higher dimensional foliations (in particular, codimension-one foliations), such as Sacksteder foliations (\cite{Sacksteder1964}).
\end{remark}



%
%

\bibliographystyle{my-amsplain-nodash-abrv-lastnamefirst-nodot}

\end{document}